\documentclass[11pt]{amsart}

\usepackage{amsmath,amsthm,amssymb,fourier,mathrsfs,enumerate,color,ifpdf}
\usepackage{tikz,tikz-cd}
\usetikzlibrary{decorations.pathmorphing,arrows,intersections}
\usepackage{enumitem}
\usepackage{float}
\usepackage{csquotes}
\usepackage[
  backend=biber,
  style=alphabetic,
  sorting=nyt,
  maxcitenames=99,
  maxalphanames=99,
  minalphanames=99,
]{biblatex}
\usepackage[colorlinks=true,
            linkcolor=red,
            citecolor=blue,
            urlcolor=blue]{hyperref}
\usepackage{listings}
\usepackage{xcolor}

\lstdefinelanguage{Sage}{
  language=Python,
  morekeywords={var,plot,matrix,Graph,graphs},
  sensitive=true
}

\DeclareLabelalphaTemplate{
  \labelelement{
    \field[strwidth=1,strside=left,uppercase]{labelname}
  }
  \labelelement{
    \field[strwidth=2,strside=right]{year}
  }
}

\DeclareCiteCommand{\cite}[\mkbibbrackets]
  {\usebibmacro{prenote}}
  {%
    \usebibmacro{citeindex}%
    \bibhyperref{%
      \printfield{labelalpha}%
      \printfield{extraalpha}%
    }%
  }
  {\addcomma\space}
  {\usebibmacro{postnote}}
\defbibenvironment{bibliography}
  {\list
     {\mkbibbrackets{\printfield{labelalpha}%
        \printfield{extraalpha}}} % left label
     {\setlength{\labelwidth}{4.5em}%
      \setlength{\leftmargin}{\labelwidth}%
      \addtolength{\leftmargin}{\labelsep}%
      \setlength{\itemsep}{\bibitemsep}%
      \setlength{\parsep}{\bibparsep}}%
  }
  {\endlist}
  {\item}

\DeclareBibliographyDriver{article}{%
  \printnames{author}%
  \newunit\newblock
  \printfield{title}%
  \newunit\newblock
  \printfield{journaltitle}%
  \setunit*{\addspace}%
  \printfield{volume}%
  \setunit{\addcolon\space}%
  \printfield{number}%
  \setunit{\addcomma\space}%
  \printfield{year}%
  \newunit\newblock
  \printfield{pages}%
  \newunit\newblock
  \printfield{doi}%
  \finentry
}

\DeclareBibliographyDriver{incollection}{%
  \printnames{author}%
  \newunit\newblock
  \printfield{title}% chapter title
  \newunit\newblock
  \printtext[emph]{\printfield{booktitle}}% book title italic
  \printtext[emph]{\printfield{note}}
  \newunit\newblock
  \printnames{editor}%
  \newunit\newblock
  \printlist{publisher}%
  \setunit{\addspace}%
  \printfield{year}%
  \newunit\newblock
  \printfield{pages}%
  \newunit\newblock
  \printfield{doi}%
  \finentry
}

\usepackage{caption}
\usepackage{subcaption}
\usepackage{comment}

\newcommand{\SL}{\operatorname{SL}}
\newcommand{\PSL}{\operatorname{PSL}}

\newcommand{\tr}{\operatorname{tr}}

\newcommand{\dist}{\operatorname{dist}}

\newcommand{\setdef}[2]{ \left\{ {#1}\ : \ {#2} \right\} }

\newtheorem{theorem}{Theorem}[section]
\newtheorem{thm}[theorem]{Theorem}
\newtheorem{mthm}{Theorem}

\newtheorem{lemma}[theorem]{Lemma}

\newtheorem{corollary}[theorem]{Corollary}

\newtheorem{proposition}[theorem]{Proposition}

\newtheorem{defn}[theorem]{Definition}

\numberwithin{equation}{section}

\begin{document}
\title{On the length sets of closed hyperbolic surfaces}

\author{Yanlong Hao}
\address{University of Michigan, Ann Arbor, Michigan, USA}
\email{ylhao@umich.edu}
\keywords{closed hyperbolic surfaces, length sets, length spectra,
trace sets, Fuchsian groups, Teichm\"uller space, closed geodesics,
algebraic exceptional sets}
\subjclass[2020]{Primary 30F35; Secondary 30F60, 53C22, 20H10}
\date{\today}

\begin{abstract}
For every closed surface $\Sigma_g$ of genus $g\geq6$, we prove
that there exists a countable union of positive-codimension
algebraic subsets of Teichm\"uller space such that, for every
hyperbolic metric $d$ outside this exceptional set, the number of
distinct primitive closed-geodesic lengths at most $L$ is bounded below by
$\tau(d)e^{\delta(g)L}$, where the explicit constant $\delta(g)$
satisfies $\delta(g)>\frac12$.
\end{abstract}

\maketitle
\section{Introduction}
Let $(\Sigma,d)$ be a closed hyperbolic surface and let
\[
\mathcal L_d^*(\Sigma)
=\{\ell_d(\gamma):\gamma\text{ is a primitive closed geodesic on }\Sigma\}
\]
be the \textit{primitive length set} of $\Sigma$, with repetitions omitted, and define
\[
N_d^*(L)=\#\left(\mathcal L_d^*(\Sigma)\cap[0,L]\right).
\]

The distinction between the counting of closed geodesics and that of distinct lengths is fundamental. By the prime geodesic theorem of Huber, the number
\[
\Pi_d(L)
=\#\left\{\gamma:
\gamma\text{ is a primitive closed geodesic and }
\ell_d(\gamma)\leq L\right\}
\]
satisfies
\[
\Pi_d(L)\sim \frac{e^L}{L}
\qquad\text{as }L\longrightarrow\infty;
\]
see \cite{MR109212}, and see \cite{MR2035655} for the dynamical formulation in negative curvature. Thus, the number of primitive closed geodesics grows exponentially. The function $N_d^*(L)$, however, counts only the distinct values assumed by their lengths, and different free homotopy classes may determine the same value.

Such coincidences cannot be avoided. Randol proved that the length spectrum of every compact hyperbolic surface has unbounded multiplicity \cite{MR553396}; see also \cite{MR1622287}. This result concerns individual length values of large multiplicity, but it does not determine the global rate of growth of the set of distinct lengths. In particular, unbounded multiplicity is compatible with the possibility that the length set itself still grows exponentially. This leads to the following geometric question:
\[
\textit{How rapidly must the set of distinct primitive closed-geodesic lengths grow?}
\]

Our main result gives an exponential lower bound for every hyperbolic metric outside an algebraically defined exceptional subset of Teichm\"uller space.

\begin{mthm}\label{thm}
Let $\Sigma_g$ be a closed surface of genus $g\geq 6$. There exists a subset
\[
\mathcal T_{\mathrm{sing}}\subset\mathcal T_g
\]
which is a countable union of algebraic subsets of positive codimension, and there exists an explicit constant $\delta(g)>\frac{1}{2}$ such that the following holds. For every
\[
d\in\mathcal T_g\setminus\mathcal T_{\mathrm{sing}},
\]
there is a constant $\tau(d)>0$ for which
\[
N_d^*(L)\geq \tau(d) e^{\delta(g)L}
\]
for all sufficiently large $L$. The constant $\delta(g)$ is defined in equation~\eqref{eq: defind of deltag}. 
\end{mthm}

There are three features of Theorem~\ref{thm} that we wish to emphasize. First, it shows that the exponentially many primitive closed geodesics cannot collapse onto only subexponentially many length values. Second, for fixed genus, the exponent $\delta(g)$ is uniform over all metrics outside the exceptional set; only the multiplicative constant $\tau(d)$ depends on the metric. Third, the genericity statement is pointwise and algebraic rather than merely measure-theoretic. The theorem holds at every point outside a countable union of proper algebraic loci defined by trace coincidences.

The gaps in the primitive length set are of interest. List $\mathcal{L}_d^*=\{\ell_i(d): i\in\mathbb{N}\}$ in increasing order. Define
\[
\operatorname{gap}_d(L):=\min\left\{
\ell_{j+1}(d)-\ell_j(d):
\ell_{j+1}(d)\leq L
\right\}.
\]
Theorem~\ref{thm} and the pigeonhole principle immediately yield an exponentially decreasing upper bound.
\begin{corollary}\label{cor1}
Let $g\geq 6$. For every
\[
d\in\mathcal T_g\setminus\mathcal T_{\mathrm{sing}},
\]
there is a constant $B_d>0$ such that
\[
\operatorname{gap}_d(L)
\leq B_d Le^{-\delta(g)L}
\]
for all sufficiently large $L$.
\end{corollary}

Thus, among the distinct primitive lengths not exceeding $L$, two consecutive values must be exponentially close. The size of gaps in length spectra has been studied from both geometric and Diophantine perspectives. For hyperbolic manifolds defined by algebraic data, arithmetic separation arguments can produce exponential lower bounds for nonzero gaps. In the opposite direction, Dolgopyat and Jakobson proved that arbitrarily small gaps occur for a topologically generic set of hyperbolic and negatively curved metrics \cite{MR3538867}. Corollary~\ref{cor1} gives a pointwise quantitative statement outside the algebraic exceptional set.

Theorem~\ref{thm} also gives quantitative information about multiplicities. Define the \textit{average multiplicity of the primitive length spectrum} up to length $L$ by
\[
\operatorname{Mult}_d(L)
=\frac{\Pi_d(L)}{N_d^*(L)}.
\]
A related, but different, notion was defined in \cite[Section 3.2]{MR1252073}. 

Theorem~\ref{thm} leads to the following.

\begin{corollary}\label{cor}
Let $g\geq6$. For every
\[
d\in\mathcal T_g\setminus\mathcal T_{\mathrm{sing}},
\]
there is a constant $A_d>0$ such that
\[
\operatorname{Mult}_d(L)
\leq A_d\frac{e^{(1-\delta(g))L}}{L}
\]
for all sufficiently large $L$. 
\end{corollary}

Since $\delta(g)>\frac12$, the exponential rate $1-\delta(g)$ in Corollary~\ref{cor} is strictly smaller than $\frac{1}{2}$. This stands in contrast with arithmetic surfaces, where arithmetic restrictions can force exponentially many primitive geodesics to be represented by a much smaller collection of length values. Such multiplicities play an important role in arithmetic quantum chaos through the geometric side of the Selberg trace formula; see, for example, \cite{MR1480112,MR1251251,MR1321639}. Exponential mean multiplicities are also known for certain semi-arithmetic surfaces \cite{MR5041863}.

We now describe the main difficulty in proving Theorem~\ref{thm}. A natural approach is to find a large convex-cocompact subgroup of the surface group and apply orbit counting. This produces exponentially many group elements, but orbit counting by itself does not produce exponentially many distinct lengths: different elements may have the same trace and hence the same translation length. The central problem is therefore to control the multiplicity of the trace map on the family of elements arising from the construction.

There are two essentially different sources of trace coincidence. The first consists of \emph{accidental coincidences}: two elements that do not have the same trace for every hyperbolic structure may nevertheless have the same trace at a particular point of Teichm\"uller space. Since squared traces are rational functions in Fricke coordinates, the coincidence locus of two such elements lies in a proper algebraic subset. The exceptional set $\mathcal T_{\mathrm{sing}}$ is obtained by taking the union of these loci over all non-trace-equivalent pairs.

The second source consists of \emph{universal trace identities}: distinct group elements may have equal squared trace for every hyperbolic structure. Such identities cannot be removed by excluding a subset of Teichm\"uller space. Instead, we analyze the relevant trace-equivalence classes directly. The short splitting construction below cuts off either a one-holed torus or a pair of pants. By applying appropriate Fenchel--Nielsen deformations in these two cases, we prove that the relevant trace-equivalence classes contain at most three elements in the one-holed-torus case and only one element in the pair-of-pants case. Consequently, outside $\mathcal T_{\mathrm{sing}}$, every squared-trace value in the family used in the proof is attained at most three times.

The geometric input is a short splitting multicurve. For every hyperbolic metric $d$ on $\Sigma_g$, we construct a multicurve $\alpha$ of explicitly bounded total length such that
\[
\Sigma_g\setminus\alpha=\Sigma\sqcup\Sigma',
\]
where $\Sigma$ is either a one-holed torus or a pair of pants. In both cases,
\[
\operatorname{area}(\Sigma)=2\pi,
\qquad
\operatorname{area}(\Sigma')=(4g-6)\pi,
\]
and
\[
\ell_d(\alpha)\leq H(g),
\]
where $H(g)$ is explicit.

Let
\[
\phi_d:\pi_1(\Sigma_g)\longrightarrow\SL(2,\mathbb R)
\]
be a lift of the holonomy representation. Set
\[
\Gamma_d'=\phi_d\bigl(\pi_1(\Sigma')\bigr)
\]
and 
\[
X_d'=\Gamma_d'\backslash\mathbf H^2
\]
be the associated complete convex-cocompact hyperbolic surface. Its convex core is isometric to $\Sigma'$, and the boundary of the convex core has total length $\ell_d(\alpha)$. Using the convex core as a test domain in the definition of the Cheeger constant gives
\[
h(X_d')
\leq
\frac{\ell_d(\alpha)}{\operatorname{area}(\Sigma')}
\leq
\frac{H(g)}{(4g-6)\pi}.
\]
The effective Buser inequality proved below therefore yields
\[
\lambda_0(X_d')
\leq
0.682h(X_d')
\leq
\frac{0.682H(g)}{(4g-6)\pi}.
\]
If
\[
\frac{0.682H(g)}{(4g-6)\pi}<\frac{1}{4},
\]
then
\[
\delta(\Gamma_d')\geq\delta(g)>\frac12.
\]
Orbit counting for the convex-cocompact group $\Gamma_d'$ consequently produces exponentially many elements of $\pi_1(\Sigma')$, with an exponential rate at least $\delta(g)$. The trace-multiplicity argument described above then converts this orbit count into the lower bound for the number of distinct primitive lengths.

For every nontrivial
$\gamma\in\pi_1(\Sigma_g)$,
\[
\left|\tr\bigl(\phi_d(\gamma)\bigr)\right|
=
2\cosh\left(\frac{\ell_d(\gamma)}{2}\right).
\]
Thus the change of variables $T\asymp e^{L/2}$ converts the primitive length-set problem into the corresponding problem for absolute traces of primitive elements. Since these trace values form a subset of the full trace set, Theorem~\ref{thm} implies Corollary~\ref{cor: Schmutz}.

\begin{corollary}\label{cor: Schmutz}
    Let $(\Sigma_g,d)$ be a closed surface of genus $g\geq 6$, and write $(\Sigma_g,d)=\Gamma_d\backslash\mathbf H^2$. If the trace set of $\Gamma_d$ has linear growth, then $d\in \mathcal{T}_{\mathrm{sing}}$.
\end{corollary}

For arithmetic Fuchsian groups, trace sets are strongly constrained by the underlying quaternion algebra. Luo and Sarnak proved the bounded-clustering property for cofinite arithmetic Fuchsian groups \cite{MR1266491}. Sarnak subsequently conjectured that bounded clustering characterizes arithmetic Fuchsian groups and that uniform positive separation of traces forces a group to be derived from a quaternion algebra \cite{MR1321639}. Schmutz proposed the stronger conjecture that a Fuchsian lattice whose trace set has linear growth must be arithmetic \cite{MR1394753}. Geninska and Leuzinger established the bounded-clustering conjecture for lattices containing parabolic elements \cite{MR2397884}. More recently, the author proved the linear-growth conjecture for nonuniform lattices and formulated a generic lower-bound problem for cocompact lattices \cite{hao2025trace}. Theorem~\ref{thm} gives the corresponding cocompact result outside an algebraically defined exceptional set.

\medskip

The broad counting strategy follows \cite[Theorem~D]{hao2025trace}: one constructs a short splitting multicurve and counts elements in the complementary subgroup. The present proof differs in two important respects. First, we obtain the splitting estimate using the construction of \cite[Section~2]{MR4905028}, whose underlying geometric idea goes back to Buser; see \cite[Section~5.2]{MR2742784}. Second, and more importantly, we replace the previously asserted pairwise distinctness of the resulting traces by a complete bounded-multiplicity argument. Universal trace identities do occur in the one-holed torus case, but their equivalence classes are uniformly bounded, while all remaining coincidences are confined to the algebraic exceptional set. This is sufficient for the orbit-counting argument and identifies precisely the mechanism by which trace distinctness may fail.

The proof of Theorem~\ref{thm} may be summarized in four steps. First, we construct a nice splitting multicurve $\alpha$ with
\[
\ell_d(\alpha)\leq H(g).
\]
Second, the complementary surface $\Sigma'$ and the Cheeger--Buser estimate give
\[
\delta(\Gamma_d')\geq\delta(g)>\frac{1}{2}.
\]
Third, Fenchel--Nielsen deformations and the algebraicity of trace functions imply that, outside $\mathcal T_{\mathrm{sing}}$, the relevant squared traces have multiplicity at most three. Finally, orbit counting in $\pi_1(\Sigma')$ produces exponentially many group elements, and the multiplicity bound converts this into exponentially many distinct primitive lengths.

The paper is organized as follows. In Section~\ref{sec: Prelimimaries}, we recall the required background. In Section~\ref{Sec: short multicurves}, we construct short splitting multicurves and obtain the explicit bound $H(g)$. In Section~\ref{sec: different traces}, we analyze trace-equivalence classes, define the algebraic exceptional set, and prove the uniform trace-multiplicity bound. Finally, in Section~\ref{Sec: proof}, we combine these ingredients with orbit counting to prove Theorem~\ref{thm}.

\section{Preliminaries}\label{sec: Prelimimaries}
We recall some background here.
\subsection{Hyperbolic geometry}
In this subsection, we recall some results on hyperbolic geometry of surfaces.

\begin{lemma}\cite{MR379833}\label{lemma: collar}
    Let $\gamma$ be a simple closed geodesic on a hyperbolic surface $(\Sigma,d)$. The neighborhood
    \[
    T(\gamma)=\{x\in\Sigma: d(x,\gamma)\leq w(\gamma)\},
    \]
    called the collar of $\gamma$, is isometric to the cylinder $(\rho, t)\in [-w(\gamma),w(\gamma)]\times \mathbb{R}/\mathbb{Z}$ with the metric
    \[
    ds^2=d\rho^2+\ell_d^2(\gamma)\cosh^2(\rho)dt^2,
    \]
    where $w(\gamma)=\sinh^{-1}\left(\frac{1}{\sinh(\frac{1}{2}\ell_d(\gamma))}\right)$.
\end{lemma}

\begin{theorem}\cite{MR2173363}\label{thm: change of boundary}
    Let $(\Sigma_{g,n},d)$ be a compact hyperbolic Riemann surface with geodesic boundary components $\gamma_1,\gamma_2,\cdots,\gamma_n$, with lengths $\ell_1,\ell_2,\cdots,\ell_n$, respectively. Let $\{\gamma_\nu\}$ be a sequence of simple closed geodesics in the interior of $\Sigma_{g,n}$. For any $\epsilon_k>0$, $k=1,\cdots, n$, there exists a compact hyperbolic Riemann surface $(\Sigma_{g,n},d')$ such that the boundary components of $d'$ are of lengths $\ell_k+\epsilon_k$, $k=1,\cdots,n$, and that $\ell_{d'}(\gamma_\nu)>\ell_d(\gamma_\nu)$ for all $\nu$.
\end{theorem}

The \textit{systole} of a closed hyperbolic surface $(\Sigma,d)$, denoted by $\operatorname{sys}(d)$, is the length of its shortest closed geodesic. Define 
\[
\operatorname{sys}(g)=\max\{\operatorname{sys}(d): d\in\mathcal{T}_g\}.
\]

Bavard's bound \cite{MR1413853} states that  
\begin{equation}\label{eq: Bavard bounds}
\operatorname{sys}(g)\leq 2\cosh^{-1}\left(\frac{1}{2\sin(\frac{\pi}{12g-6})}\right) \leq 2\ln g+2.681.
\end{equation}
Recent work improves this bound \cite{bourque2023linear}.

By \cite{MR749104},
\[
\operatorname{sys}(2)=2\cosh^{-1}(1+\sqrt2)>2
\]
is realized by the Bolza surface.
For every \(g\geq2\), the Bolza surface admits a connected cover of
degree \(g-1\), which has genus \(g\). Since the systole does not
decrease under coverings, it follows that
\[
\operatorname{sys}(g)\geq\operatorname{sys}(2)>2.
\]

We will also use the right-angled pentagon and hexagon formulas \cite[Equations (2.6.10) and (2.6.17)]{MR1435975}.
\subsection{Critical exponent and orbit counting}
Let $\Gamma$ be a Fuchsian group, and fix a point $o\in {\mathbf{H}}^2$. The \textit{critical exponent} of $\Gamma$ is the abscissa of convergence of the \textit{Poincar\'e series} of $\Gamma$:
\[ 
P_\Gamma(s)=\sum_{\gamma\in \Gamma}e^{-s\dist_{\mathbf{H}^2}(o,\gamma o)},
\]
where $\dist_{\mathbf{H}^2}$ is the hyperbolic metric on ${\mathbf{H}}^2$.

The critical exponent is closely related to the \textit{orbit counting function} of the group; in particular, we have:
\begin{theorem}\cite{MR948288}\label{Thm: critical}
   Suppose $\Gamma$ is a convex cocompact subgroup of $\PSL(2,\mathbb{R})$. Then there exists $c>0$ and a $\Gamma$-invariant continuous function $F: {\mathbf{H}}^2\to \mathbb R^+$, such that 
   \[
   \#\{\gamma\in\Gamma\mid \dist_{\mathbf{H}^2}(x,\gamma x')\leq R\}\sim cF(x)F(x')e^{\delta_\Gamma R}
   \]
   for all $x$, $x'\in{\mathbf{H}}^2$.
\end{theorem}

We also discuss the relationship between $\lambda_0$ of ${\mathbf{H}}^2/\Gamma$ and the critical exponent $\delta_\Gamma$ of a Fuchsian group $\Gamma$, where $\lambda_0$ is the bottom of the spectrum of the Laplace-Beltrami operator.

\begin{thm}\cites{MR360472, MR360473, MR450547}\label{cheeger}
    Let $\Gamma$ be a Fuchsian group. Then 
    \[
    \lambda_0({\mathbf{H}}^2/\Gamma)=\left\{\begin{array}{lc}
    \frac{1}{4}&\delta_\Gamma<\frac{1}{2},\\
    \delta_\Gamma(1-\delta_\Gamma)& \delta_\Gamma\geq \frac{1}{2}.
    \end{array}
    \right.
    \]
\end{thm}
This result has been generalized to discrete subgroups of $\mathrm{PSO}(n,1)$ in \cite{MR1074486}.

\subsection{Cheeger constant and spectral estimates}
Introduced by Cheeger \cite{MR402831} in 1970, the \textit{Cheeger constant} is a crucial invariant in geometry and spectral theory. For any compact $n$-dimensional Riemannian manifold, define the Cheeger constant of $M$ by
\[
h(M)=\inf\frac{\mathrm{Vol}(\partial A)}{\mathrm{Vol}(A)},
\]
where $A$ runs over all open subsets with $\mathrm{Vol}(A)\leq \frac{1}{2}\mathrm{Vol}(M)$. Cheeger \cite{MR402831} showed $h(M)$ provides a lower bound of $\lambda_0$:
\[
\lambda_0(M)\geq \frac{1}{4}h^2(M).
\]
Cheeger's inequality also holds for noncompact Riemannian manifolds, provided that $h$ is defined over all $A$ such that $A\cup\partial A$ is compact.

Indeed, $h(M)$ also gives an upper bound of $\lambda_0$.
\begin{theorem}\cite[Theorem 7.1]{MR683635}\label{buser}
    There exists a constant $\nu$ such that for any noncompact hyperbolic surface $\Sigma$,
    \[
    \lambda_0(\Sigma)\leq \nu h(\Sigma).
    \]
\end{theorem}
 We require the following effective form of this estimate.

\begin{lemma}\label{lemma: better kappa}
     For an infinite-volume complete hyperbolic surface, the $\nu$ in Theorem~\ref{buser} may be taken to be 0.682.
\end{lemma}
\begin{proof}
    Let $S$ be an infinite-volume complete hyperbolic surface and let $\lambda_0=\lambda_0(S)$, $h=h(S)$. Since $\operatorname{Ric}(S)=-1$, \cite[Theorem 1.1]{MR4225812} implies that 
    \[
    h\geq \sup_{t>0}\sqrt{2\pi}\frac{1-e^{-\lambda_0t}}{\tanh^{-1}(\sqrt{1-e^{-2t}})}=\sup_{t>0}\sqrt{2\pi}\frac{1-e^{-\lambda_0t}}{\cosh^{-1}(e^{t})}.
    \]
    It follows that 
    \[
    \frac{\lambda_0}{h}\leq \frac{\lambda_0\cosh^{-1}(e^t)}{\sqrt{2\pi}(1-e^{-\lambda_0t})}.
    \]
    Since $\frac{\lambda_0}{1-e^{-\lambda_0t}}$ is increasing and $\lambda_0\leq \frac{1}{4}$ by Theorem~\ref{cheeger}, we have
    \[
    \frac{\lambda_0}{h}\leq \frac{\cosh^{-1}(e^t)}{4\sqrt{2\pi}(1-e^{-\frac{1}{4}t})}.
    \]
    Setting $t=2$ gives
    \[
    \frac{\lambda_0}{h}\leq 0.681483\ldots\leq 0.682.
    \]
\end{proof}

\subsection{Fricke coordinates} Let $\Sigma_g$ be a closed surface of genus $g$. The fundamental group of $\Sigma_g$ has the presentation
\[
\Gamma_g=\langle A_1,B_1,A_2,B_2,\cdots, A_g,B_g: \prod_{i=1}^g[A_i,B_i]=e\rangle.
\]
The \textit{Fricke coordinates} for the Teichm\"uller space of $\Sigma_g$ are the collection of real numbers $\underline{X}:=(a_i, c_i, d_i)_{i=1}^{2g-2}$, where $c_i>0$ for all $i$. For each $\underline{X}$, there is a corresponding representation $\psi_{\underline{X}}:\Gamma_g\to \PSL(2,\mathbb{R})$  given by 
\[
\psi_{\underline{X}}(A_i)=\begin{pmatrix}
    a_{2i-1}&\frac{a_{2i-1}d_{2i-1}-1}{c_{2i-1}}\\
    c_{2i-1}& d_{2i-1}
\end{pmatrix}, \quad 1\leq i\leq g-1;
\]
\[
\psi_{\underline{X}}(B_i)=\begin{pmatrix}
    a_{2i}&\frac{a_{2i}d_{2i}-1}{c_{2i}}\\
    c_{2i}& d_{2i}
\end{pmatrix}, \quad 1\leq i\leq g-1;
\]
\[
\psi_{\underline{X}}(A_{g})=\begin{pmatrix}
    a_g&b_g\\
    c_g& d_g
\end{pmatrix}, \quad a_g+d_g=b_g+c_g>0;
\]
\[
\psi_{\underline{X}}(B_{g})=\begin{pmatrix}
    \zeta&0\\
    0& \frac{1}{\zeta}
\end{pmatrix}, \quad \zeta>1.
\]
The numbers $a_g$, $b_g$, $c_g$, $d_g$, and $\zeta$ are uniquely determined up to a sign by the Fricke coordinates and the fundamental relation 
\[
\prod_{i=1}^g[A_{i},B_{i}]=e.
\]

This representation admits a lift to $\SL(2,\mathbb{R})$. This is a classical result; see, for instance, \cite[Lemma 2.2]{MR4920633}. By abuse of notation, we regard  $\psi_{\underline{X}}$ as an $\SL(2,\mathbb{R})$-valued representation. For details, see \cite[p.~49]{MR1215481} or \cite[Section 8.1]{MR4874183}. 

For a point $d\in\mathcal{T}_g$, we also denote the corresponding representation associated with its Fricke coordinates by $\psi_d$. By \cite[Section 8.2]{MR4874183}, for every $\gamma\in \Gamma_g$, the function $\mathrm{tr}^2(\psi_{\underline{X}}(\gamma))$ is  rational in the coordinates $\underline{X}$.

\section{Short splitting multicurves on closed surfaces}\label{Sec: short multicurves}
In this section, we apply the construction in \cite{MR4905028} to form the splitting multicurve used in the next section.
\begin{defn}
Let $\Sigma_g$ be a surface of genus $g\geq 2$. A multicurve $\alpha$ is called a \emph{nice splitting multicurve} if $\Sigma_g\setminus\alpha$ is the disjoint union of two connected surfaces $\Sigma$ and $\Sigma'$, where $\Sigma$ has type $(0,3)$ or $(1,1)$.
\end{defn}
We now construct a short nice splitting multicurve, following the method of \cite{MR4905028}. Because our required estimate differs from the one used there, we include the details.

Let $(\Sigma_g,d)$ be a closed surface of genus $g\geq 2$. Note that a homologically nontrivial simple closed curve is nonseparating. By \cite{MR2893492}, there is always a nonseparating curve $\gamma$ on $\Sigma_g$ with length less than $\operatorname{sys}(g)$. Then $\Sigma_\gamma:=\Sigma_g\setminus\gamma$ is a surface of type $(g-1,2)$. 

We first construct a splitting curve on surfaces of type $(g-1,2)$ with geodesic boundary curves of length $\operatorname{sys}(g)$.

The following construction is used in \cite[p.~409]{MR4905028}. We use a different estimate.

For $i=1,2$, let $w_i$ be the maximal number such that the half-collar 
\[
N(\gamma_i,w_i):=\setdef{x\in \Sigma_\gamma}{d(x,\gamma_i)<w_i}
\]
is isometric to $[0,w_i)\times S^1$ with the metric of Lemma~\ref{lemma: collar}. 

\textbf{Case 1: The two collar neighborhoods do not meet.} In this case, \[\partial N(\gamma_1,w_1)\cap \partial N(\gamma_2,w_2)=\emptyset.\]
Now there are geodesic arcs $C_i$ on $\Sigma_\gamma$ connecting two different points on $\gamma_i$ of length $2w_i$ for $i=1,2$. Let $\alpha_1=\{\gamma_1, \mu_1, \mu_2\}$ be the multicurve shown in Figure~\ref{figure1}. Define $\alpha_2$ similarly. We call the replacement of $\gamma_i$ by $\alpha_i$ a \textit{type I move}. The multicurve $\alpha_i$  cuts off a pair of pants $Y_i$ from $\Sigma_\gamma$.  
\begin{figure}[ht]

\begin{tikzpicture}[
    line cap=round,
    line join=round,
    thick,
    scale=1.05
]
\begin{scope}[xscale=0.5]
    
% ------------------------------------------------
% Left: pair of pants
% ------------------------------------------------

% left boundary circle
\draw (-0.2,0) ellipse [x radius=0.38, y radius=0.88];

% top boundary of the pair of pants / attached surface
\draw
  (-0.2,0.88)
  .. controls (1.1,0.80) and (2.8,1.02) .. (4.2,1.02)
  .. controls (4.9,1.02) and (5.35,1.55) .. (5.95,1.62)
  -- (9.25,1.92);

\draw
  (-0.2,-0.88)
  .. controls (1.1,-0.80) and (2.8,-1.02) .. (4.2,-1.02)
  .. controls (4.9,-1.02) and (5.35,-1.55) .. (5.95,-1.62)
  -- (9.25,-1.92);

% bottom boundary of the pair of pants / attached surface
%\draw (-0.2,-0.88).. controls (1.5,-0.88) and (3.3,-0.90) .. (4.55,-1.06).. controls (5.20,-1.15) and (5.75,-1.40) .. (6.05,-1.72).. controls (6.60,-1.60) and (7.35,-1.58) .. (8.65,-1.35)-- (9.25,-1.35);

% ------------------------------------------------
% Middle: central ellipse and seams
% ------------------------------------------------

% central ellipse (full ellipse, not half-circle)
\draw (7.10,0.18) ellipse [x radius=1.05, y radius=0.34];

% short neck joining left part to the ellipse
%\draw (4.95,0.18).. controls (5.45,0.18) and (5.75,0.18) .. (6.05,0.18);

% dashed gluing seams to the right-hand surface
%\draw[dashed] (6.25,1.92) -- (6.25,0.52);
%\draw[dashed] (6.25,-0.06) -- (6.25,-1.72);

% solid half
    \draw[dashed] (6.25,-0.06) arc[start angle=90,end angle=270,x radius=0.2,y radius=0.8];

    % dashed half
    \draw (6.25,-1.66) arc[start angle=-90,end angle=90,x radius=0.2,y radius=0.8];

    \draw[dashed] (6.25,1.62) arc[start angle=90,end angle=270,x radius=0.2,y radius=0.6];

    % dashed half
    \draw (6.25,0.42) arc[start angle=-90,end angle=90,x radius=0.2,y radius=0.6];

% small connecting curves near the ellipse
%\draw (5.75,0.52).. controls (5.92,0.42) and (6.00,0.30) .. (6.05,0.18);

%\draw (5.75,-0.16).. controls (5.92,-0.06) and (6.00,0.06) .. (6.05,0.18);

% ------------------------------------------------
% Geodesic
% ------------------------------------------------

\draw[dashed] (-0.55,0.2) arc[start angle=180,end angle=0,x radius=3.3,y radius=0.1];

\draw (0.18,-0.05) arc[start angle=200,end angle=360,x radius=3.05,y radius=0.22];
% ------------------------------------------------
% Optional labels
% ------------------------------------------------
\node at (8.0,2.18) {$\Sigma_2$};
\node[below] at (3.3,0.18) {$C_1$};
\node at (7,1.22) {$\mu_1$};
\node at (7,-0.84) {$\mu_2$};
\node at (-1,0) {$\gamma_1$};
\end{scope}
\end{tikzpicture}
\caption{Type I move for $\gamma_1$}\label{figure1}
\end{figure}
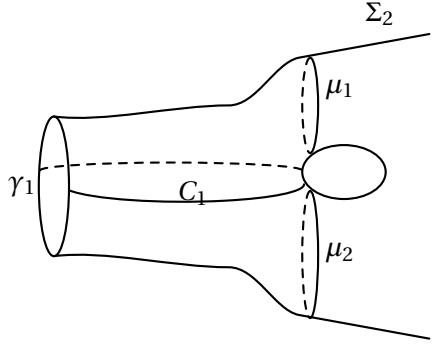

The estimate of the length of $\alpha_i$ in \cite{MR4905028} does not suffice for our application. Hence we will use a more robust estimate. 

Since the half collar $N(\gamma_i,w_i)$ is inside $\Sigma_\gamma$, the area computation leads to \[2\min_i\{\operatorname{sys}(g)\sinh(w_i)\}\leq (4g-4)\pi.\] 
Hence,
\begin{equation}\label{eq: w}
\min_i\{w_i\}\leq \ln\left(1+\frac{4(g-1)\pi}{\operatorname{sys}(g)}\right).
\end{equation}

The three geodesics that minimize the distance between the boundaries of the pair of pants and $C_i$ cut it into four right-angled pentagons as in Figure~\ref{fig: alpha}. 

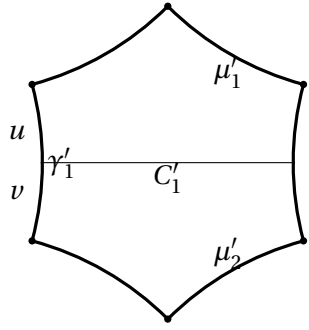
\begin{figure}[ht]
    \centering
\begin{tikzpicture}[scale=4]

% Parameters for a regular right-angled hyperbolic hexagon
\pgfmathsetmacro{\rho}{sqrt(2-sqrt(3))} % vertex radius
\pgfmathsetmacro{\C}{sqrt(2)}            % geodesic circle center radius

% Six hyperbolic sides.
% Each side is an arc of a Euclidean circle orthogonal to the unit circle.
\foreach \a in {0,60,120,180,240,300} {
  \draw[very thick]
    ({\C*cos(\a) + cos(\a+195)},
     {\C*sin(\a) + sin(\a+195)})
    arc[start angle={\a+195}, end angle={\a+165}, radius=1];
}

% Vertices
\foreach \a in {30,90,150,210,270,330} {
  \fill ({\rho*cos(\a)}, {\rho*sin(\a)}) circle (0.012);
}
\draw (-0.42, 0)--(0.42, 0);
\node at (0,-0.05) {$C_1'$};
\node at (-0.5,0.1) {$u$};
\node at (-0.5,-0.1) {$v$};
\node at (-0.35,0) {$\gamma'_1$};
\node at (0.2,0.3) {$\mu'_1$};
\node at (0.2,-0.3) {$\mu'_2$};

\end{tikzpicture}
\caption{Computation of Length of $\alpha$}
    \label{fig: alpha}
\end{figure}
In the figure, $C_1'$ cuts the curve $\gamma_1'$ into two pieces with lengths $u$ and $v$ respectively. We also have
\[
\ell_d(C_1')=w_1, \quad \ell_d(\mu_1')=\frac{\ell_d(\mu_1)}{2},\quad \ell_d(\mu_2')=\frac{\ell_d(\mu_2)}{2}, \quad u+v=\frac{\operatorname{sys}(g)}{2}.
\]
By the right-angled pentagon formula \cite[Equation (2.6.17)]{MR1435975},
\[
\cosh(\ell_d(\mu'_1))=\sinh(u)\sinh(w_1),
\]
\[
\cosh(\ell_d(\mu'_2))=\sinh(v)\sinh(w_1).
\]
Hence, \[2(\ell_d(\mu'_1)+\ell_d(\mu'_2))=2
\left(\cosh^{-1}(\sinh u\sinh{w_1})+\cosh^{-1}(\sinh v\sinh{w_1})\right).\] 
Since $\cosh^{-1}(\sinh t \sinh{w_1})$ is a concave function for $t>0$, we have
\begin{equation}\label{eq: type I move}
\ell_d(\mu_1)+\ell_d(\mu_2)\leq 4\cosh^{-1}\left(\sinh\frac{\operatorname{sys}(g)}{4}\sinh{w_1}\right)\leq \operatorname{sys}(g)+4w_1-4\ln 2.\end{equation}
The latter inequality follows from the facts
\[
\sinh\frac{\operatorname{sys}(g)}{4}\sinh{w_1}\leq \frac{1}{4}\exp\left(\frac{\operatorname{sys}(g)}{4}+w_1\right),\quad 
\cosh^{-1}(z)\leq \ln(2z).
\]

The same argument applies to $\gamma_2$. It leads to 
\begin{equation}\label{eq: length of alpha}
\ell_d(\alpha_i)\leq 2\left(\operatorname{sys}(g)+2w_i-2\ln2\right).
\end{equation}

Without loss of generality, we may assume $w_1\leq w_2$. Together with equation~\eqref{eq: w}, equation~\eqref{eq: length of alpha} implies that 
\[
\ell_d(\alpha_1)\leq 2\operatorname{sys}(g)-4\ln 2+4\ln \left(1+\frac{4(g-1)\pi}{\operatorname{sys}(g)}\right).
\]
Define 
\[
h_{\mathrm{I}}(g):=2\operatorname{sys}(g)-4\ln 2+4\ln \left(1+\frac{4(g-1)\pi}{\operatorname{sys}(g)}\right).
\]
Then $\alpha_1$ is a nice splitting multicurve with total length at most $h_{\mathrm{I}}(g)$.

\textbf{Case 2: The two collar neighborhoods meet.} In this case, \[\partial N(\gamma_1,w_1)\cap \partial N(\gamma_2,w_2)\neq\emptyset.\] 
We are in the situation of Figure~\ref{Figure 2}.
Set $\alpha=\eta$. The curve $\alpha$ is a nice splitting curve.
We call this replacement a type II move.
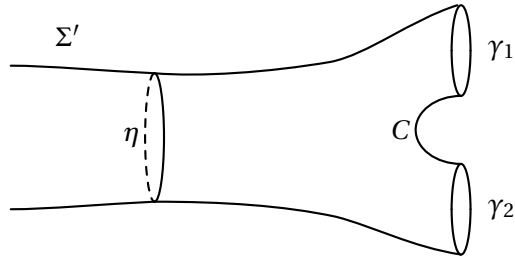
\begin{figure}[h]
\begin{tikzpicture}[
    line cap=round,
    line join=round,
    thick,
    scale=1.00
]

% ------------------------------------------------
% Left: ambient surface Sigma_2
% ------------------------------------------------
\begin{scope}[xscale=0.5]
% top boundary of the surface
\draw
  (-3.8,0.95)
  .. controls (-2.8,0.95) and (-1.3,0.88) .. (0,0.85);

% bottom boundary of the surface
\draw
  (-3.8,-0.95)
  .. controls (-2.8,-0.95) and (-1.3,-0.88) .. (0,-0.85);

\node at (-2.3,1.35) {$\Sigma'$};

% ------------------------------------------------
% Glued boundary (half dashed, half solid ellipse)
% ------------------------------------------------

% left half dashed
\draw[dashed] (0,0.85)
  arc[start angle=90,end angle=270,x radius=0.25,y radius=0.85];

% right half solid
\draw (0,-0.85)
  arc[start angle=-90,end angle=90,x radius=0.25,y radius=0.85];

\node[left] at (-0.15,0) {$\eta$};

% ------------------------------------------------
% Pair of pants to the right
% ------------------------------------------------
%\begin{scope}[xscale=0.5]
    
% outer top boundary
\draw
  (0,0.85)
  .. controls (1.2,0.80) and (3.0,0.86) .. (4.65,1.00)
  .. controls (5.65,1.10) and (6.75,1.55) .. (8,1.75);

% outer bottom boundary
\draw
  (0,-0.85)
  .. controls (1.8,-0.85) and (3.4,-0.90) .. (4.80,-1.05)
  .. controls (5.80,-1.18) and (6.80,-1.45) .. (8.1,-1.55);

\draw (8.1,0.55)
  arc[start angle=90,end angle=270,x radius=1.2,y radius=0.45];  

\node at (6.5,0.1) {$C$};

% ------------------------------------------------
% Right: two free boundary components
% ------------------------------------------------

% top free boundary gamma_0
\draw (7.85,1.15)
  arc[start angle=180,end angle=0,x radius=0.25,y radius=0.60];
\draw (8.35,1.15)
  arc[start angle=0,end angle=-180,x radius=0.25,y radius=0.60];

% bottom free boundary gamma_1
\draw (7.85,-0.95)
  arc[start angle=180,end angle=0,x radius=0.25,y radius=0.60];
\draw (8.35,-0.95)
  arc[start angle=0,end angle=-180,x radius=0.25,y radius=0.60];

\node[right] at (8.55,1.15) {$\gamma_1$};
\node[right] at (8.55,-0.95) {$\gamma_2$};

% ------------------------------------------------
% Curve c inside the pants
% ------------------------------------------------
\end{scope}
\end{tikzpicture}
\caption{Type II move}\label{Figure 2}
\end{figure}

By the right-angled hexagon formula \cite[Equation~(2.6.10)]{MR1435975},
\[
\ell_d(\eta)=2\cosh^{-1}\left(\cosh(\ell_d(C))\sinh^2\left(\frac{\ell_d(\gamma_1)}{2}\right)-\cosh^2\left(\frac{\ell_d(\gamma_1)}{2}\right)\right)
\]
The same estimate gives 
\[
\ell_d(\eta)\leq 2\ell_d(\gamma_1)+2\ell_d(C)-4\ln2.
\]
Since $\ell_d(C)=w_1+w_2$,  
\[
\ell_d(\eta)\leq 2\left(\operatorname{sys}(g)-2\ln2+w_1+w_2\right).
\]

Since the half collar $N(\gamma_i,w_i)$ is inside $\Sigma_\gamma$, the area computation leads to \[\ell_d(\gamma_1)\sinh(w_1)+\ell_d(\gamma_2)\sinh(w_2)\leq (4g-4)\pi.\] 
The function $\sinh^{-1}(t)$ is concave for $t>0$. Then by Jensen's inequality,
\[
w_1+w_2\leq 2\sinh^{-1}\left(\frac{2(g-1)\pi}{\operatorname{sys}(g)}\right).
\]
Hence,
\begin{equation}\label{eq: eta}
    \ell_d(\eta)\leq 2\left(\operatorname{sys}(g)-2\ln2+2\sinh^{-1}\left(\frac{2(g-1)\pi}{\operatorname{sys}(g)}\right)\right)
\end{equation}

Define 
\[
h_{\mathrm{II}}(g)=2\left(\operatorname{sys}(g)-2\ln2+2\sinh^{-1}\left(\frac{2(g-1)\pi}{\operatorname{sys}(g)}\right)\right),
\]
and set $H_{\mathrm{split}}(g)=\max\{h_{\mathrm{I}}(g),h_{\mathrm{II}}(g)\}.$ In either of the two cases above, we obtain a nice splitting multicurve of total length at most $H_{\mathrm{split}}(g)$.

Since for all $t>0$,
    \[
    \sinh^{-1}(t)\leq \ln(1+2t),
    \]
    $h_{\mathrm{I}}(g)\geq h_{\mathrm{II}}(g)$. Therefore $H_{\mathrm{split}}(g)=h_{\mathrm{I}}(g)$.

    The function $f(t)=2t+4\ln(1+\frac{4(g-1)\pi}{t})$ is increasing on $(2,\infty)$ because
    \[
    f'(t)=2+\frac{4}{t+4(g-1)\pi}-\frac{4}{t}>2-\frac{4}{2}=0.\] 
    Since $\operatorname{sys}(g)\geq 2$, by equation~\eqref{eq: Bavard bounds},
    \[
    H_{\mathrm{split}}(g)=h_{\mathrm{I}}(g)\leq 4\ln g+5.362-4\ln2+4\ln\left(1+\frac{4(g-1)\pi}{2\ln g+2.681}\right).
    \]

    Define
    \begin{equation}\label{eq: define h}
    H(g)=4\ln g+5.362-4\ln2+4\ln\left(1+\frac{4(g-1)\pi}{2\ln g+2.681}\right).
    \end{equation}

Now we are back to the general case. Let $(\Sigma_g,d)$ be a closed hyperbolic surface of genus $g$. Recall that by \cite{MR2893492}, there is always a nonseparating curve $\gamma$ on $\Sigma_g$ with 
\[\ell_d(\gamma)\leq\operatorname{sys}(g).\] 
If
\[
\ell_d(\gamma)=\operatorname{sys}(g),
\]
then the preceding construction applies directly and yields the required nice splitting multicurve. We may therefore assume that
\[
\ell_d(\gamma)<\operatorname{sys}(g).
\] 
Let $d_0$ denote the metric induced on the surface obtained by cutting $(\Sigma_g,d)$ along $\gamma$. Its two boundary components both have length $\ell_d(\gamma)<\operatorname{sys}(g)$.

Apply Theorem~\ref{thm: change of boundary} to an enumeration of the isotopy classes of all essential simple closed curves in the interior of the cut surface. We obtain a metric $d_+$ with the
prescribed boundary lengths such that
\[
\ell_{d_0}(\beta)\leq\ell_{d_+}(\beta)
\]
for every essential simple closed curve $\beta$ in the interior.

We perform the collar construction with respect to $d_+$. Let $\alpha$ be the resulting nice splitting multicurve. Every component of $\alpha$ is either one of the distinguished boundary curves or an essential simple closed curve in the interior. Consequently,
\[
\ell_{d_0}(\alpha)\leq\ell_{d_+}(\alpha).
\]

Regard the components of $\alpha$ as topological curves in the original cut surface, and reglue the two copies of $\gamma$. The resulting curves form a nice splitting multicurve $\alpha$ on $(\Sigma_g,d)$.

 The preceding discussion proves the following result.

\begin{proposition}\label{Propo: nice splitting multicure}
    Let $d\in \mathcal{T}_g$. There exists a nice splitting multicurve on $(\Sigma_g,d)$ of total length at most $H(g)$.
\end{proposition}

\section{Trace-equivalence classes and bounded trace multiplicity}\label{sec: different traces}
In this section, we construct families of elements whose squared traces have uniformly bounded multiplicity for generic hyperbolic structures.

Let $(\Sigma_g,d)$ be a closed surface of genus $g$. By Proposition~\ref{Propo: nice splitting multicure}, there is a nice splitting multicurve $\alpha$ of total length at most $H(g)$. $\Sigma_g\setminus \alpha$ is a disjoint union of two surfaces $\Sigma$ and $\Sigma'$. By the definition of nice splitting multicurves, $\Sigma$ is of type $(1,1)$ or $(0,3)$. We deal with the two different cases in the next two subsections.
\subsection{One-holed torus} In this subsection, we assume that $\Sigma$ is a surface of type $(1,1)$. Then $\Sigma'$ is a surface of type $(g-1,1)$. Let $\pi_1(\Sigma)=F(a_1,b_1)$ and $\pi_1(\Sigma')=F(a_2,b_2,\cdots,a_g,b_g)$. Then
\[
\pi_1(\Sigma_g)=\langle a_1,b_1,\cdots,a_g,b_g:\prod_{i=1}^g[a_i,b_i]=e \rangle.
\]

Any $d\in \mathcal{T}_g$ induces a faithful representation 
\[
\phi_d:\pi_1(\Sigma_g)\to \SL(2,\mathbb{R}).
\]

For each $A\in\pi_1(\Sigma_g)$, define the squared trace function
\[
T_A:\mathcal T_g\longrightarrow\mathbb R,
\qquad
T_{A}(d)=\tr^2(\phi_d(A)).
\]
The use of the squared trace makes this function independent of the choice of lift of $\phi_d$ to $\SL(2,\mathbb R)$. 

\begin{defn}
Two elements $A,B\in\pi_1(\Sigma_g)$ are called \emph{trace equivalent}, denoted by
\[
A\sim_{\tr}B,
\]
if
\[
T_A(d)=T_B(d)
\]
for every $d\in\mathcal T_g$. Furthermore, for $A,X,Y\in\pi_1(\Sigma')$, $X$ and $Y$ are \emph{$A$-trace equivalent} if
\[
AX\sim_{\tr} AY.
\]
\end{defn}

Thus, trace equivalence is an identity between trace functions on the whole Teichm\"uller space, rather than an equality occurring at one particular hyperbolic structure.

We analyze the $a_1$-trace-equivalent classes in $\pi_1(\Sigma')$. Notice that, unlike in \cite{hao2025trace}, there are distinct elements $X, Y\in \pi_1(\Sigma')$ that are $a_1$-trace equivalent. Indeed, let \[C=\prod_{i=2}^g[a_i,b_i]=[b_1,a_1].\] Then \[a_1C=a_1b_1a_1b_1^{-1}a_1^{-1}\] is conjugate to $a_1$. Thus, $C$ and $e$ form such a pair. We will show later that this is essentially the only obstruction.

Let $X, Y\in\pi_1(\Sigma')$ such that $a_1X$ and $a_1Y$ are trace equivalent.

We fix one $d\in\mathcal{T}_g$ and a representation $\phi_d$ such that \[
\phi_d(C)=\begin{pmatrix}
    -\kappa&0\\
    0&-\tfrac{1}{\kappa}
\end{pmatrix}\]
for some $\kappa>1$. Note that $\operatorname{tr}(\phi_d(C))<-2$; see \cite[Section 3A]{MR454075}.
Define \[C(s)=\begin{pmatrix}
    \kappa^s&0\\
    0&\kappa^{-s}
\end{pmatrix}.\] 

Let $\phi_d(a_1)=\begin{pmatrix}
    a_{11}&a_{12}\\
    a_{21}&a_{22}
\end{pmatrix}.$ Since $\operatorname{tr}(\phi_d(a_1))=\operatorname{tr}(\phi_d(a_1C))$, we have 
\[
-\kappa a_{11}-\frac{1}{\kappa} a_{22}=a_{11}+a_{22}.
\]
It follows that 
\begin{equation}\label{eq: diagonal of a1}
    a_{22}=-\kappa a_{11}.
\end{equation}
If $a_{11}=a_{22}=0$, then $\phi_d(a_1^2)$ will commute with $\phi_d(C)$, which is not the case since $\phi_d$ is faithful and $a_1^2$ does not commute with $C$. We conclude that $a_{11}\neq0$. The condition $\det \phi_d(a_1)=1$ implies that \begin{equation}\label{eq: skew diagonal of a1}
a_{12}a_{21}=a_{11}a_{22}-1<-1.\end{equation}
%Similarly, there exist $P\in \SL(2,\mathbb{R})$ and $\tau>1$ such that \[\phi_d(b_1)=P\begin{pmatrix}\tau&0\\0&\tfrac{1}{\tau}\end{pmatrix}P^{-1}.\]Define $B(s)=P\begin{pmatrix}\tau^s&0\\0&\tau^{-s}\end{pmatrix}P^{-1}.$

Consider the following family of representations: 
\[
\phi_d^{s}:\pi_1(\Sigma_g)\to \SL(2,\mathbb{R})
\]
given by 
\[
\phi_d^{s}(a_1)=C(s)\phi_d(a_1)C(-s); \quad \phi_d^{s}(b_1)=C(s)\phi_d(b_1)C(-s
);\]
and for $2\leq i\leq g$,
\[
\phi^s_d(a_i)=\phi_d(a_i); \quad \phi^s_d(b_i)=\phi_d(b_i).
\]
Geometrically, $\phi_d^s$ corresponds to the Fenchel–Nielsen twist along the curve $C$.

Since $\mathcal{T}_g$ is an open set in the character variety, for any $s$ in a neighborhood $U$ of $0$, $\phi_d^{s}$ is a discrete, faithful representation, and provides a point in $\mathcal{T}_g$. Since the relevant trace functions are continuous and nonzero on $U$, the sign relating the two traces is constant on $U$ It follows that 
\[
\phi_d^{s}(a_1X)+\phi_d^{s}(a_1Y)\quad \text{or}\quad \phi_d^{s}(a_1X)-\phi_d^{s}(a_1Y)
\]
has zero trace for all $s$ in this neighborhood. Choose the sign so that
\[M^\pm=\phi_d(X)\pm \phi_d(Y)=\phi_d^{s}(X)\pm \phi^s_d(Y)\] satisfies the trace identity. Then we have for all $s\in U$,
\[
C(s)\phi_d(a_1)C(-s)M^\pm
\]
has trace zero. This condition imposes a strong restriction on the matrix $M^\pm$.

Let $M^\pm=\begin{pmatrix}
    m_{11}&m_{12}\\
    m_{21}&m_{22}
\end{pmatrix}$. Then
\begin{equation}
    C(s)\phi_d(a_1)C(-s)M^\pm=\begin{pmatrix}
       a_{11}m_{11}+\kappa^{2s}a_{12}m_{21}&*\\
       *&a_{22}m_{22}+\kappa^{-2s}a_{21}m_{12}
    \end{pmatrix}.
\end{equation}
The trace condition implies
\[
a_{11}m_{11}+a_{22}m_{22}=0; \quad a_{12}m_{21}=0; \quad a_{21}m_{12}=0.
\]
Using equations~\eqref{eq: diagonal of a1}, \eqref{eq: skew diagonal of a1} and $a_{11}\neq 0$, we obtain
\[
m_{11}-\kappa m_{22}=0;\quad m_{12}=0; \quad m_{21}=0.
\]
Hence $M^\pm=t\begin{pmatrix}
    \kappa&0\\
    0&1
\end{pmatrix}$ for some $t\in \mathbb{R}$. 

Let $\phi_d(Y)=\begin{pmatrix}
    y_{11}&y_{12}\\
    y_{21}&y_{22}
\end{pmatrix}$.
It follows that for some $t\in \mathbb{R}$
\[
\phi_d(X)=t\begin{pmatrix}
    \kappa&0\\
    0&1
\end{pmatrix}\mp\phi_d(Y).
\]
Taking the determinant, we have 
\[
t(t\kappa\mp(y_{11}+\kappa y_{22}))=0.
\]
Therefore, for each choice of sign, there is at most one nonzero value of $t$ satisfying the equation. This proves the following lemma.
\begin{lemma}\label{lemma: trace of case 1}
    Suppose that $\Sigma$ has type $(1,1)$. Then every $a_1$-trace-equivalent class in $\pi_1(\Sigma')$ contains at most three elements.
\end{lemma}
\subsection{Pair of pants.} In this subsection, we assume that $\Sigma$ is a surface of type $(0,3)$. Then $\Sigma'$ is a surface of type $(g-2,3)$. Let $\pi_1(\Sigma')=F(b_1,b_2,a_3,b_3,\cdots,a_g,b_g)$. Then there exist $a_1$ and $a_2$ such that
\[
\pi_1(\Sigma_g)=\langle a_1,b_1,\cdots,a_g,b_g: \prod_{i=1}^g[a_i,b_i]=e \rangle.
\]

We analyze $a_1^{-1}a_2$-trace-equivalent classes in $\pi_1(\Sigma')$. 

Let $X, Y\in\pi_1(\Sigma')$ be an $a_1^{-1}a_2$-trace-equivalent pair. 

Fix one $d\in\mathcal{T}_g$ and a representation $\phi_d$ such that \[\phi_d(b_1)=\begin{pmatrix}
    \lambda &0\\
    0&\frac{1}{\lambda}
\end{pmatrix}, \quad \lambda>1.\]  Define $b_1(s)=\begin{pmatrix}
    e^s &0\\
    0&e^{-s}
\end{pmatrix}$. There exist $P\in \SL(2,\mathbb{R})$ and $\eta>1$ such that
\[\phi_d(b_2)=P\begin{pmatrix}
    \eta &0\\
    0&\frac{1}{\eta}
\end{pmatrix}P^{-1}.\] Define $b_2(s)=P\begin{pmatrix}
    e^s &0\\
    0&e^{-s}
\end{pmatrix}P^{-1}$.

Consider the family of representations
\[
\phi_d^{s,t}:\pi_1(\Sigma_g)\to \SL(2,\mathbb{R})
\]
given by 
\[
\phi_d^{s,t}(a_1)=\phi_d(a_1)b_1(s); \quad \phi_d^{s,t}(a_2)=\phi_d(a_2)b_2(t);
\]
\[
\phi_d^{s,t}(Z)=\phi_d(Z) \text{ for all $Z\in\pi_1(\Sigma')$.}\]

Geometrically, as above, $\phi_d^{s,t}$ corresponds to the Fenchel–Nielsen twists along the curves $b_1$ and $b_2$.

Since $\mathcal{T}_g$ is an open set in the character variety, for any $(s,t)$ in a neighborhood $V$ of $(0,0)$, $\phi_d^{s,t}$ is a discrete, faithful representation, and provides a point in $\mathcal{T}_g$. As in the one-holed torus case, the relevant traces are
continuous and nonzero on $V$. Hence, the sign relating the two traces is constant on $V$. It follows that  
\[
\phi_d^{s,t}(a_1^{-1}a_2X)+\phi_d^{s,t}(a_1^{-1}a_2Y)\quad \text{or}\quad \phi_d^{s,t}(a_1^{-1}a_2X)-\phi_d^{s,t}(a_1^{-1}a_2Y)
\]
has zero trace for all $(s,t)$ in this neighborhood. Choose the sign so that
\[M^\pm=\phi_d(X)\pm \phi_d(Y)=\phi_d^{s,t}(X)\pm \phi_d^{s,t}(Y)\] satisfies the trace identity. Then we have for all $(s,t)\in V$,
\[
b_1(-s)\phi_d(a_1^{-1}a_2)Pb_1(t)P^{-1}M^\pm
\]
has trace zero. This condition imposes a strong restriction on the matrix $M^\pm$.

We first show that every entry of $N=\phi_d(a_1^{-1}a_2)P$ is nonzero. Let \[N=\begin{pmatrix}
   n_{11}&n_{12}\\
   n_{21}&n_{22}
\end{pmatrix}.\] 

Note that $Nb_1(\ln \eta)N^{-1}=\phi_d(a_1^{-1}a_2b_2a_2^{-1}a_1)$ is  conjugate to $\phi_d(b_2)$. 

If $n_{11}n_{12}=0$, then this element and $\phi_d(b_1)$ have a common fixed point $0$ in
$\partial\mathbf H^2$. Both elements are hyperbolic. Since they
belong to a discrete Fuchsian group,
\cite[Theorem~5.1.2]{MR1393195} implies that they have the same two fixed points. After conjugating these fixed points to
$0$ and $\infty$, the subgroup generated by the two elements is identified, through the logarithm of the multiplier, with a discrete subgroup of $(\mathbb R,+)$. It is therefore cyclic. Moreover, $b_1$ is primitive in $\pi_1(\Sigma_g)$, so $\phi_d(b_1)$ generates this cyclic stabilizer. Consequently,
\[
\phi_d(a_1^{-1}a_2b_2a_2^{-1}a_1)=\phi_d(b_1)^k
\]
for some $k\in\mathbb Z$. The faithfulness of the holonomy
representation then gives
\[
a_1^{-1}a_2b_2a_2^{-1}a_1=b_1^k.
\]
Passing to homology yields
\[
[b_2]=k[b_1],
\]
contradicting the linear independence of $[b_1]$ and $[b_2]$ in
$H_1(\Sigma_g;\mathbb Z)$. Thus $n_{11}n_{12}\neq0$. The same
argument applied to the other fixed point gives
$n_{21}n_{22}\neq0$. Hence every entry of $N$ is nonzero.

Let $Q=P^{-1}M^\pm=\begin{pmatrix}
   q_{11}&q_{12}\\
   q_{21}&q_{22}
\end{pmatrix}$. Then
$b_1(-s)\phi_d(a_1^{-1}a_2)Pb_1(t)P^{-1}M^\pm$ has the form\[\begin{pmatrix}
   e^{-s+t}n_{11}q_{11}+e^{-s-t}n_{12}q_{21}&*\\
   *& e^{s+t}n_{21}q_{12}+e^{s-t}n_{22}q_{22}
\end{pmatrix}.
\]
Since the trace vanishes for every $(s,t)\in V$ 
and the four functions
\[
e^{-s+t},\quad e^{-s-t},\quad e^{s+t},\quad e^{s-t}
\]
are linearly independent, while every $n_{ij}$ is nonzero, we obtain
\[
q_{11}=q_{12}=q_{21}=q_{22}=0.
\] Hence $Q=0$. Therefore, $M^\pm=PQ=0$. Thus, $\phi_d(X)=\pm \phi_d(Y)$. After projecting to $\PSL(2,\mathbb R)$, the two holonomy images agree. Faithfulness therefore gives $X=Y$.

\begin{lemma}\label{lemma: trace of case 2}
    Suppose that $\Sigma$ has type $(0,3)$. Then every $a_1^{-1}a_2$-trace-equivalent class in $\pi_1(\Sigma')$ contains exactly one element.
\end{lemma}

\begin{defn}
    In either case, we call
\[
\{a_i,b_i:1\leq i\leq g\}
\]
the \emph{system of generators associated with $d$}.
\end{defn}
\subsection{The algebraic exceptional set}\label{subsec: exceptional set}

The exceptional set $\mathcal T_{\mathrm{sing}}$ was defined in
\cite[Section 8.2]{MR4874183}. We recall its definition and the
property needed below.

By the discussion of Fricke coordinates in Section~\ref{sec: Prelimimaries}, for any $A\in \pi_1(\Sigma_g)$, the function $T_A$ is a rational function in the Fricke coordinates.

Suppose that $A$ and $B$ are not trace equivalent. Then
\[
F_{A,B}
:=
T_A-T_B
\]
is a nonzero rational function in the Fricke coordinates. After taking a common denominator, we may write
\[
F_{A,B}
=\frac{P_{A,B}}{Q_{A,B}},
\]
where $P_{A,B}$ and $Q_{A,B}$ are polynomials, $P_{A,B}$ is not identically zero, and $Q_{A,B}$ does not vanish on the relevant Fricke-coordinate domain.

Define the coincidence locus
\[
\mathcal Z_{A,B}=\setdef{d\in\mathcal T_g}{T_A(d)=T_B(d)}.
\]
Equivalently,
\[
\mathcal Z_{A,B}=\setdef{d\in\mathcal T_g}{P_{A,B}(d)=0}.
\]
Since $P_{A,B}$ is not the zero polynomial, every nonempty locus $\mathcal Z_{A,B}$ is the intersection of $\mathcal T_g$ with a proper algebraic subset of the Fricke-coordinate space. In particular, it has positive codimension in $\mathcal T_g$.

We define
\[
\mathcal T_{\mathrm{sing}}
=\bigcup_{\substack{A,B\in\pi_1(\Sigma_g)\\
A\not\sim_{\tr}B}}
\mathcal Z_{A,B}.
\]

\begin{proposition}\label{prop: exceptional set}
The subset $\mathcal T_{\mathrm{sing}}\subset\mathcal T_g$ is a countable union of proper algebraic subsets of positive codimension. Moreover, if
\[
d\in\mathcal T_g\setminus\mathcal T_{\mathrm{sing}},
\]
then, for any $A,B\in\pi_1(\Sigma_g)$,
\[
T_A(d)=T_B(d)
\]
if and only if $A$ and $B$ are trace equivalent.
\end{proposition}

\begin{proof}
The group $\pi_1(\Sigma_g)$ is countable, and hence the collection of pairs
\[
(A,B)\in\pi_1(\Sigma_g)\times\pi_1(\Sigma_g)
\]
is countable. For every pair that is not trace equivalent, the corresponding nonempty coincidence locus $\mathcal Z_{A,B}$ is a proper algebraic subset of positive codimension. Therefore, $\mathcal T_{\mathrm{sing}}$ is a countable union of such subsets.

Now let $d\notin\mathcal T_{\mathrm{sing}}$. Suppose that
\[
T_A(d)=T_B(d).
\]
If $A$ and $B$ were not trace equivalent, then by definition
\[
d\in\mathcal Z_{A,B}
\subset\mathcal T_{\mathrm{sing}},
\]
which is a contradiction. Hence $A\sim_{\tr}B$. The converse follows immediately from the definition of trace equivalence.
\end{proof}

We apply this observation to the trace-equivalence relations introduced above. Let $\alpha$ be a nice splitting multicurve, let $\Sigma'$ be the corresponding complementary subsurface, and let $A_d$ denote the element associated with $d$: namely,
$A_d=a_1$
in the one-holed torus case and
$A_d=a_1^{-1}a_2$
in the pair-of-pants case. For $X,Y\in\pi_1(\Sigma')$, recall that $X$ and $Y$ are $A_d$-trace equivalent if
\[
A_dX\sim_{\tr} A_dY.
\]

\begin{corollary}\label{cor: bounded trace multiplicity}
Let
\[
d\in\mathcal T_g\setminus\mathcal T_{\mathrm{sing}}.
\]
If $X,Y\in\pi_1(\Sigma')$ satisfy
\[
\tr^2\!\left(\phi_d(A_dX)\right)
=
\tr^2\!\left(\phi_d(A_dY)\right),
\]
then $X$ and $Y$ are $A_d$-trace equivalent.

Consequently, every squared-trace value in the family
\[
\setdef{\tr^2\!\left(\phi_d(A_dX)\right)}{X\in\pi_1(\Sigma')}
\]
is attained at most three times in the one-holed torus case and at most once in the pair of pants case.
\end{corollary}

\begin{proof}
The equality
\[
\tr^2\!\left(\phi_d(A_dX)\right)
=
\tr^2\!\left(\phi_d(A_dY)\right)
\]
means that
\[
T_{A_dX}(d)=T_{A_dY}(d).
\]
Since $d\notin\mathcal T_{\mathrm{sing}}$, Proposition~\ref{prop: exceptional set} implies that
\[
A_dX\sim_{\tr}A_dY.
\]
Thus $X$ and $Y$ are $A_d$-trace equivalent. The cardinality bounds now follow from Lemmas~\ref{lemma: trace of case 1} and~\ref{lemma: trace of case 2}.
\end{proof}

The important point is that $\mathcal T_{\mathrm{sing}}$ removes only accidental trace coincidences. Trace identities that hold for every hyperbolic structure are not included in the exceptional set; instead, they are controlled by the uniform bounds on the $A_d$-trace-equivalent classes established in the preceding subsections.

Finally, we show that all elements in this family are primitive.
\begin{lemma}\label{lem: primitive}
For every $X\in\pi_1(\Sigma')$, the element $A_dX$ is not a proper power. Consequently, it determines a primitive closed geodesic on $(\Sigma_g,d)$.
\end{lemma}

\begin{proof}
The homology class
\[
[A_dX]\in H_1(\Sigma_g;\mathbb Z)\cong\mathbb Z^{2g}
\]
is a primitive vector. Indeed, in the one-holed-torus case its
$[a_1]$-coordinate is $1$, while in the pair-of-pants case its
$[a_1]$- and $[a_2]$-coordinates are $-1$ and $1$, respectively.
If $A_dX=Y^n$ for some $|n|\geq2$, then
\[
[A_dX]=n[Y],
\]
contradicting the primitivity of $[A_dX]$.
\end{proof}

\section{Proof of Theorem~\ref{thm}}\label{Sec: proof}
\begin{proof}[Proof of Theorem~\ref{thm}]
Let
\[
d\in\mathcal T_g\setminus\mathcal T_{\mathrm{sing}}.
\]
By Proposition~\ref{Propo: nice splitting multicure}, there exists a nice splitting multicurve $\alpha$ on $(\Sigma_g,d)$ of total length at most $H(g)$. Write
\[
\Sigma_g\setminus\alpha=\Sigma\sqcup\Sigma',
\]
where $\Sigma$ has type $(1,1)$ or $(0,3)$.

Choose the corresponding system of generators as in Section~\ref{sec: different traces}. If $\Sigma$ has type $(1,1)$, set
\[
A_d=a_1.
\]
If $\Sigma$ has type $(0,3)$, set
\[
A_d=a_1^{-1}a_2.
\]

Finally, define
\[
\Gamma'_d=\phi_d\bigl(\pi_1(\Sigma')\bigr),
\qquad
X'_d=\Gamma'_d\backslash\mathbf H^2.
\]

By Corollary~\ref{cor: bounded trace multiplicity}, for all $X,Y\in \pi_1(\Sigma')$ which are not $A_d$-trace equivalent,
\[\tr^2\left(\phi_d(A_dX)\right)\neq \tr^2\left(\phi_d(A_dY)\right).\]

Fix a basepoint $o\in \mathbf{H}^2$ and $R>0$. Let \[N'(d,R)=\setdef{Z\in \pi_1(\Sigma')}{\dist_{\mathbf{H}^2}(\phi_d(Z)o,o)\leq R}.\]
Recall that $N_d^*(R)=\#(\mathcal{L}_d^*(\Sigma_g)\cap [0,R])$.

Since for all $X\in \pi_1(\Sigma')$, \[\ell_d(A_dX)\leq \dist_{\mathbf{H}^2}\left(\phi_d(A_dX)o,o\right)\leq \dist_{\mathbf{H}^2}\left(\phi_d(A_d)o,o\right)+\dist_{\mathbf{H}^2}\left(\phi_d(X)o,o\right),\] for all $Z\in N'(d,R)$, $\ell_{d}(A_dZ)\leq R+\dist_{\mathbf{H}^2}\left(\phi_d(A_d)o,o\right)$.  

By Corollary~\ref{cor: bounded trace multiplicity} and Lemma~\ref{lem: primitive}, 
\[
N_d^*(R)\geq \frac{1}{3}\#N'\left(d,R-\dist_{\mathbf{H}^2}\left(\phi_d(A_d)o,o\right)\right).
\]
Note that $\Gamma'_d$ is a convex cocompact subgroup with convex core $\Sigma'$. Theorem~\ref{Thm: critical} then implies that there exists a constant $\tau(d)>0$ such that for $R$ sufficiently large,
\[
N_d^*(R)\geq \tau(d)e^{\delta(\Gamma'_d)R},
\]
where $\delta(\Gamma'_d)$ is the critical exponent of $\Gamma'_d$.

To estimate the critical exponent of $\Gamma'_d$, note that the boundary of $\Sigma'$ is $\alpha$. Hence the Cheeger constant 
\[
h(X'_d)\leq \frac{\ell_d(\alpha)}{(4g-6)\pi}.
\]
By Theorem~\ref{buser} and Lemma~\ref{lemma: better kappa}, 
\[
\lambda_0(X'_d)\leq \frac{0.682\ell_d(\alpha)}{(4g-6)\pi}.
\]

In view of Theorem~\ref{cheeger}, $\delta(\Gamma'_d)>\frac{1}{2}$ if and only if $\lambda_0(X'_d)<\frac{1}{4}$. By Proposition~\ref{Propo: nice splitting multicure}, it is sufficient to show that 
\[
\frac{0.682H(g)}{(4g-6)\pi}<\frac{1}{4}.
\]

\textbf{Claim:} This inequality is valid for $g\geq 6$.

We will prove the claim at the end of the section.

For $g\geq 6$, let $\delta(g)$ be the larger root of the following equation
\[
x(1-x)=\frac{0.682H(g)}{(4g-6)\pi}.
\]
That is, 
\begin{equation}\label{eq: defind of deltag}
\delta(g)=\frac{1}{2}\left(1+\sqrt{1-\frac{2.728H(g)}{(4g-6)\pi}}\right),
\end{equation}
where $H(g)$ is given by equation~\eqref{eq: define h}. 
Then $\delta(\Gamma'_d)\geq \delta(g)>\frac{1}{2}$.

The result follows.
\end{proof}
Before the proof of the claim, we remark that for $g$ sufficiently large, $H(g)\leq 8\ln g$. Consequently, for all sufficiently large $g$, one has the simpler estimate
\[\delta(g)>1-\frac{5.456\ln g}{(2g-3)\pi}.\]
\begin{proof}[Proof of the claim]
    Let 
    \[
    F(t)=4\ln t+5.362-4\ln2+4\ln\left(1+\frac{4(t-1)\pi}{2\ln t+2.681}\right)-\frac{(4t-6)\pi}{2.728}.
    \]
    Let $a(t)=\ln t+1.3405$. Then
    \[
    \frac{1}{4}F'(t)=\frac{1}{t}+\frac{\frac{1}{t}+2\pi}{a(t)+2\pi(t-1)}-\frac{1}{ta(t)}-\frac{\pi}{2.728}.
    \]
    For $t\geq 3$, $a(t)>0$. Therefore
    \[
    \frac{1}{4}F'(t)\leq \frac{1}{t}+\frac{\frac{1}{t}+2\pi}{2\pi(t-1)}-\frac{\pi}{2.728}\leq \frac{1}{3}+\frac{\frac{1}{3}+2\pi}{4\pi}-\frac{\pi}{2.728}<0.
    \]
    Therefore $F(t)$ is decreasing on $(3,\infty)$. Direct computation shows that $F(5)>0$ and $F(6)<0$. Thus $F(t)<0$ for all $t\geq 6$. 
    
    Note that 
    \[
\frac{0.682H(g)}{(4g-6)\pi}<\frac{1}{4}
\]
if and only if $F(g)<0$.

    The result follows.
\end{proof}
\printbibliography
\end{document}